%% file: main.tex
\documentclass{article}

\usepackage[
  left=1.25in,
  right=1.25in,
  top=1.2in,
  bottom=1.2in
]{geometry}

\usepackage{amsmath, amssymb, amsthm}
\usepackage{mathtools}
\usepackage{hyperref}
\usepackage{dsfont}

\usepackage{pgfplots}
\pgfplotsset{compat=1.18}

\usepackage[dvipsnames]{xcolor}
\usepackage{algorithm}
\usepackage{bbm}
\usepackage{enumitem}
\usepackage{algpseudocode}
\usepackage{bm}
\usepackage[section]{placeins}
\usepackage{tikz}
\usepackage{cancel}
\usepackage[style= numeric-comp,hyperref=true, doi=false,url=false,
            isbn=false,
            firstinits=true, sorting = none, 
            block=none, backend=bibtex,maxnames=99]{biblatex}

\bibliography{graphon}

\usetikzlibrary{arrows.meta,calc,positioning,fit,decorations.pathreplacing}

\newcommand{\N}{\mathbb{N}}
\newcommand{\R}{\mathbb{R}}

\newcommand{\cX}{\mathsf{X}}
\newcommand{\bfo}{\mathbf{1}}

\newcommand{\cY}{\mathsf{Y}}
\newcommand{\cC}{\mathsf{C}}
\newcommand{\cD}{\mathsf{D}}

\newtheorem{definition}{Definition}
\newtheorem{theorem}{Theorem}

\newtheorem{proposition}[theorem]{Proposition}

\newtheorem{lemma}{Lemma}

\newtheorem{remark}{Remark}
\newtheorem{example}{Example}

\newcommand{\cE}{\mathcal{E}}

\newcommand{\xc}[1]{\vspace{.1cm}
\noindent {\em #1} }

\title{
Largest $2$-regular Subgraphs in complete $S$-partite Graphs
}

\author{Yiyang Jiang$^*$ \quad and \quad Xudong Chen\footnote{Y. Jiang and X. Chen are with the Department of Electrical \& Systems Engineering, Washington University, St. Louis, MO 63130, USA.  Emails: {\tt\small  \{j.yiyang, cxudong\}@wustl.edu}. Corresponding author: Y. Jiang.
}%
}
\date{}
\begin{document}

\maketitle

\begin{abstract}
In this paper, we focus on the class of complete $S$-partite graphs, for $S$ an undirected graph possibly with self-loops, and address the problem of finding largest $2$-regular subgraphs of these graphs, which can be formulated as an integer linear program. Roughly speaking, a complete $S$-partite graph is obtained by replacing every single node of $S$ with a number of nodes, preserving the edge/non-edge relations of $S$. Our motivation in finding largest $2$-regular subgraphs is rooted in the structural systems theory, particularly in the problem of finding largest subnetworks that can sustain controllability or asymptotic stability of the corresponding subsystems. A main contribution of the paper is to show that the integer linear problem can be solved efficiently in $O(|V(S)|^3)$, independent of the order/size of the $S$-partite graph itself. Furthermore, we demonstrate through simulations that with high probability, a random $S$-partite graph contains a largest $2$-regular subgraph of the same order as its complete counterpart does. 
\end{abstract}

\tableofcontents

\section{Introduction}
Let $G$ be an undirected graph, with $V(G)$ the node set and $E(G)$ the edge set. 
A subgraph $H$ of $G$ is {\it $2$-regular} if it is a {\it node-wise disjoint union of cycles}. If, further, $V(H) = V(G)$, then $H$ is a cycle cover of $G$. 

In this paper, we focus on a special class of graphs, namely, the complete $S$-partite graphs, for $S$ an undirected graph on $q$ nodes possibly with self-loops. We address the problem of finding largest $2$-regular subgraphs in these graphs. Roughly speaking, a complete $S$-partite graph can be obtained from $S$ by replacing each single node of $S$ with a number of nodes (blowing up). A precise definition will be given soon in Section~\ref{sec:mainresult}. 
A major contribution of the paper is to show that the order of a largest $2$-regular subgraph can be obtained efficiently in $O(q^3)$, independent of the order/size of the $S$-partite graph.

\subsection{Motivation and relevance of our problem} 
Our interest in finding $2$-regular subgraphs is rooted in the structural systems theory, which is about understanding which graphs can sustain a given system property, such as controllability and stability. To elaborate, consider an $n$-dimensional linear time-invariant system $\dot x(t) = Ax(t)$. We say that the matrix $A$ is {\it compatible with $G$} if $ a_{ij} \neq 0 \Rightarrow (v_i,v_j) \in E(G)$, where $a_{ij}$ is the $ij$th entry of $A$. 
It has been shown in~\cite{kirkoryan2014decentralized} that there exists a Hurwitz matrix $A$ compatible with $G$ if and only if ({\it i}) $G$ (more precisely, the corresponding symmetric digraph $\vec G$) has a cycle cover, and ({\it ii}) every node of $G$ is connected to a self-loop. 
If $G$ is connected, then the second condition will be satisfied if $G$ has at least one self-loop. 
Now, consider the case where $G$ does not have a cycle cover and hence, there does not exist a Hurwitz matrix $A$ compatible with $G$. Then, instead of asymptotically stabilizing the entire state $x(t)$ (which is infeasible in this case), one may choose to stabilize a sub-state (i.e., a collection of the $x_i(t)$'s for $i \in \{1,\ldots, n\}$). {\it What is the maximal dimension of the sub-state?}  
This question can essentially be translated to the one of finding the largest $2$-regular subgraph of $G$.

Besides asymptotic stability, another fundamental system property that requires $G$ to have a cycle cover is (ensemble) controllability~\cite{chen2021sparse} for linear time-invariant control systems. In this setting, finding a largest $2$-regular subgraph of $G$ is key to finding a controllable subspace of maximal dimension.

Cycle covers also arise naturally in problems relevant to routing and network design. For example, in the traveling salesman problem, one aims to seek a minimum-cost, closed tour that visits every node exactly once and returns to its starting point~\cite{laporte1992traveling}. In graph terms, such a tour is a {\it Hamilton cycle}, hence a connected spanning $2$-regular subgraph. 
If one keeps the degree-two condition at every node but drops the requirement that the selected edges form a single connected tour, one obtains a cycle cover, or equivalently a {\it $2$-factor} in an undirected graph. Because minimum-weight cycle covers can be computed efficiently, they are often used as intermediate structures in approximation algorithms, where one first constructs a cycle cover and then patches its subtours into a single tour~\cite{manthey2009cyclecovers}. Similar ideas also appear when only an appropriate subset of nodes is required to belong to a cycle. For example, in the ring star problem~\cite{labbe2004ringstar}, one aims to seek a cycle, together with an assignment that links the nodes outside the cycle to nodes within, in order to balance the routing and assignment costs.



The $S$-partite graphs are natural objects for modeling networks that have multiple communities of different sizes, where the non-edges of $S$ represent the cyber- and/or physical-constraints on the connections between different communities (e.g., social networks, biological networks, and infrastructure networks such as transportation or power-grid networks). 
Note, in particular, that random graphs $G$ sampled from stochastic block models~\cite{holland1983stochastic,abbe2018community} or step-graphons $W$ (i.e., graphons that are step-functions)~\cite{lovasz2006limits,belabbas2023geometric} are $S$-partite graphs, where $S$ is the skeleton graph associated with the step-graphon. In the latter case, the random graphs $G$ are dense, and these dense random graphs $G$ behave essentially the same as the complete $S$-partite graphs from the perspective of embedding bounded-degree subgraphs (in our context, $2$-regular subgraphs) into them. 
For details, we refer the reader to the Blow-up Lemma~\cite{komlos1997blow} and its use in determining Hamiltonicity of step-graphons~\cite{chen2025hamiltonicity}. Numerical study will also be carried out in Section~\ref{sec:ns} for demonstration. 


\subsection{Literature review}

At the graph-theoretic level, our problem is exactly the {\it largest $2$-regular subgraph problem} on the restricted host class of complete $S$-partite graphs: for any host graph $G$ and any vertex set $U \subseteq V(G)$, the subgraph of $G$ induced by $U$ admits a cycle cover if and only if there exists a $2$-regular subgraph $H \subseteq G$ with $V(H)=U$. In the terminology of~\cite{Pulleyblank12}, this objective can also be viewed as a {\it maximum $2$-factor problem}, namely, finding a collection of node-wise disjoint cycles covering as many nodes as possible. A representative work in this direction is~\cite{CKKPW19}, which studies largest $2$-regular subgraphs in $3$-regular graphs and derives sharp lower bounds on their size in terms of cut-edges. A different and stricter line is the {\it maximum $r$-regular induced subgraph problem}; for $r=2$, one asks for a largest node set whose induced subgraph is itself $2$-regular~\cite{GRS12}. In our setting, we only require the retained induced complete $S$-partite subgraph to admit a cycle cover.

A closely related problem is the Hamilton-cycle problem, where the spanning $2$-regular subgraph is required to be a single cycle. When the host graph is not Hamiltonian, a natural relaxation is the longest-cycle problem, which asks for a cycle of maximum length. These questions have also been studied in multipartite settings; see, for example,~\cite{HamKpartite} for Hamiltonian cycles in $k$-partite graphs and~\cite{LCTripartite} for long cycles in balanced tripartite graphs. This perspective is close in spirit to ours, in that one asks how much of the graph can still support a prescribed cyclic structure.

Finally, we note that Hamiltonicity of step-graphons has been addressed in~\cite{belabbas2021h,belabbas2023geometric,GAO20257,chen2025hamiltonicity}. A step-graphon $W$ is said to be weakly (resp., strongly) Hamiltonian if the random graph $G$ sampled from $W$ has a cycle cover (resp., a Hamilton cycle) asymptotically almost surely. A complete characterization of their Hamiltonicity has been obtained in these works.   


\subsection{Notation} 
Let $e_1,\ldots,e_q$ be the standard basis of $\R^q$, and $\bfo$ be the vector of all ones of the appropriate dimension, which will be clear in the context. 
For $x,y\in\R^q$ we write $y\le x$ to indicate that $y_i\le x_i$ for all $i = 1,\ldots, q$. The $\ell^1$-norm is $\|x\|_1 := \sum_{i=1}^q |x_i|$.  We use $\N$ (resp., $\N_0$) to denote the set of positive (resp., nonnegative) integers, and $\R_{\ge0}$ to denote the set of nonnegative real numbers. 
For a graph $G$, let $E^*(G)$ be the space of real-valued functions on $E(G)$. Note that $E^*(G)$ can be identified with the Euclidean space $\R^{|E(G)|}$.  

\section{Problem Formulation and Main Result}\label{sec:mainresult}
We start by introducing the (complete) $S$-partite graphs: 

\begin{definition}[$S$-partite graph]\label{def:spartite}
    Let $S$ be an undirected graph on $q$ nodes, possibly with self-loops. An undirected graph $G$ is an {\bf $S$-partite graph} if there exists a graph homomorphism $\pi: G \to S$. Further, $G$ is a {\bf complete $S$-partite graph} if $$(v_i,v_j)\in  E(G)  \quad \Longleftrightarrow \quad (\pi(v_i),\pi(v_j)) \in E(S).$$
\end{definition}

For a given vector $x = (x_1,\ldots, x_q) \in \N^q_0$, we let $K_x$ be the complete $S$-partite graph, with $|\pi^{-1}(u_i)| = x_i$ for all $u_i\in V(S)$. 
The problem we address and solve in this paper is the following:
\begin{equation}\label{eq:mainproblem}
\begin{aligned}
\max\;& |V(H)| \\
\text{s.t.}\;& H \subseteq K_x \text{ is $2$-regular.}
\end{aligned}
\end{equation}

We introduce below a few key objects. For each edge $f_j=(u_k,u_\ell)\in E(S)$, allowing $u_k= u_\ell$, let $z_j:=\tfrac12(e_k+e_\ell)$ be the {\it incidence vector} of $f_j$. Let $m:= |E(S)|$ be the number of edges of $S$, and $Z\in \R^{q\times m}$ be the {\it incidence matrix} of $S$: 
\begin{equation}\label{eq:incidencematrix}
Z:= 
\begin{bmatrix}
z_1 & \cdots & z_m 
\end{bmatrix}.
\end{equation}
Further, let $\cX$ be the convex cone generated by the $z_j$'s, i.e.,  
\begin{equation}\label{def:edgecone}
  \cX := \left \{\sum_{j = 1}^m c_j z_j \mid c_j \geq 0 \quad \mbox{for all } j =1,\ldots, m \right \}.
\end{equation}

The main result of the paper, which we present below, relates the solution of~\eqref{eq:mainproblem} to the solution of the following linear program:
\begin{equation}\label{eq:LP-y}
\begin{aligned}
\max \;& \mathbf 1^{\!\top}y\\  
\text{s.t. } &  y\in\cX  \quad \mbox{and} \quad y \le x. 
\end{aligned}
\end{equation}
Let 
\begin{equation}\label{eq:defPx}
\cY_x:= \{y\in \cX \mid y \mbox{ solves~\eqref{eq:LP-y}}\}.
\end{equation}
Note that $\cY_x$ is a polytope for any given $x\in \N^q$.  
A point $y\in \cY_x$ is said to be an {\it extremal solution} if $y$ is a vertex of the polytope. 

The main result of the paper is then the following:  

\begin{theorem}\label{thm:main}
   Let $x\in \N^q$ and $S$ be an undirected graph on $q$ nodes, possibly with self-loops. Then, the following hold:
   \begin{enumerate}
   \item Every extremal solution to~\eqref{eq:LP-y} is integer valued.
   \item An integer solution to~\eqref{eq:LP-y} can be obtained in time $O(q^3)$.
   \item Suppose that $x_i \geq 3$ for all $i = 1,\ldots, q$; then, for any integer solution $y$ to~\eqref{eq:LP-y},  the complete $S$-partite graph $K_y$ has a  cycle cover $H$ with at most $q$ cycles. Moreover, any such subgraph $H$ is a solution to~\eqref{eq:mainproblem}.  
   \end{enumerate}
\end{theorem}


The remainder of the section is organized as follows: Section~\ref{sec:items12} establishes Items~1 and~2 of Theorem~\ref{thm:main} by reformulating the linear program~\eqref{eq:LP-y} as an auxiliary optimization problem on a bipartite graph. Section~\ref{sec:KyHam} develops a Hamiltonicity result for complete $S$-partite graphs, which is key to proving Item~3 in Section~\ref{sec:thm1item3}. Section~\ref{sec:ns} presents numerical studies on sampled $S$-partite graphs. The paper ends with conclusions.

\section{Proof of Items~1 and~2 of Theorem~\ref{thm:main}}\label{sec:items12}

\subsection[Half-integrality of the linear program]%
{Half-integrality of the Linear Program}\label{sec:half-int}

Note that $y\in \cX$ holds if and only if there exists a vector $c\in \R^m_{\geq 0}$ such that $y = Zc$. 
Since each incidence vector $z_j$ is a probability vector (i.e., its column sum is one), $\bfo^\top Z = \bfo^\top$ 
(note that the two vectors of ones are of different dimensions). It follows that the linear program~\eqref{eq:LP-y} can be re-formulated as
\begin{equation}\label{eq:P-c}
\begin{aligned}
\max \;& \mathbf 1^{\!\top}c \\[2pt]
\text{s.t. } & Z c \le x \quad \mbox{and} \quad c  \geq  0.
\end{aligned}
\end{equation}
In particular, the maximal value of~\eqref{eq:P-c} is the same as the maximal value of~\eqref{eq:LP-y}.
Similar to~\eqref{eq:defPx}, we introduce the solution set  of~\eqref{eq:P-c} as
$$
\cC_x:= \left \{c \in \R^m \mid c \mbox{ solves~\eqref{eq:P-c}} \right \},
$$
which is a polytope. The goal of this section is to establish the following result:

\begin{proposition}\label{prop:ytoc}
    Every vertex $c$ of $\cC_x$ is integer valued. 
\end{proposition}

Note that if $c$ is a solution to~\eqref{eq:P-c}, then $y = Zc$ is a solution to~\eqref{eq:LP-y}; this holds because $\bfo^\top c = \bfo^\top y$. The map $c \in \cC_x \mapsto Zc \in \cY_x$ is clearly surjective. Given a vertex $y \in \cY_x$, there must exist a vertex $c\in \cC_x$ such that $y = Zc$. 
Thus, an immediate consequence of Proposition~\ref{prop:ytoc} is that any vertex of $\cY_x$ is {\it half-integer valued}, i.e., each entry $y_i$ is a  multiple of $1/2$. 




The remainder of the subsection is to establish Proposition~\ref{prop:ytoc}. 
To this end, consider the bipartite graph $B$ derived from $S$ as follows:  
The node set $V(B)$ is given by $V(B) = V'(B) \sqcup V''(B)$ where
\begin{equation*}
V'(B) := \{u'_1,\ldots, u'_q\}  \mbox{ and } V''(B):=
\{u''_1,\ldots, u''_q\}. 
\end{equation*}
The edge set $E(B)$ is given by
\begin{equation*}
\begin{aligned}
E(B) :=\;& \{(u'_i,u''_j), (u'_j, u''_i) \mid (u_i, u_j)\in E(S)\}. 
\end{aligned}
\end{equation*}
See Figure~\ref{fig:StoB} for illustration.

\begin{figure}[ht]
  \centering
  \begin{tikzpicture}[
      >=Latex,
      line width=0.9pt,
      dot/.style={circle,inner sep=0pt,minimum size=4pt,fill=black},
      baseedge/.style={shorten >=4pt,shorten <=4pt,draw=black!55},
      loophl/.style={shorten >=4pt,shorten <=4pt,draw=blue!70,line width=1.15pt},
      pairhl/.style={shorten >=4pt,shorten <=4pt,draw=green!55!black,line width=1.15pt},
      every label/.style={font=\small}
    ]

    \begin{scope}[scale=1.5,local bounding box=Sfig]
      \coordinate (ui) at (0,0.3);
      \coordinate (uj) at (2,0.3);
      \coordinate (uk) at (1,-0.3);

      \draw[baseedge] (ui) -- (uj);
      \draw[pairhl]  (uj) -- (uk);
      \draw[baseedge] (uk) -- (ui);
      \draw[loophl,loop left,distance=16pt,in=110,out=70] (ui) to (ui);

      \node[dot,label=below:$u_i$] at (ui) {};
      \node[dot,label=below:$u_j$] at (uj) {};
      \node[dot,label=below:$u_k$] at (uk) {};
    \end{scope}

    \begin{scope}[xshift=6.5cm,scale=1.5,local bounding box=Bfig]
      \coordinate (uiL) at (0,  0.6);
      \coordinate (ujL) at (0,  0.0);
      \coordinate (ukL) at (0, -0.6);
      \coordinate (uiR) at (2,  0.6);
      \coordinate (ujR) at (2,  0.0);
      \coordinate (ukR) at (2, -0.6);

      \draw[loophl]  (uiL) -- (uiR);
      \draw[baseedge] (uiL) -- (ujR);
      \draw[baseedge] (ujL) -- (uiR);
      \draw[pairhl]   (ujL) -- (ukR);
      \draw[pairhl]   (ukL) -- (ujR);
      \draw[baseedge] (ukL) -- (uiR);
      \draw[baseedge] (uiL) -- (ukR);

      \node[dot,label=left :$u'_i$] at (uiL) {};
      \node[dot,label=left :$u'_j$] at (ujL) {};
      \node[dot,label=left :$u'_k$] at (ukL) {};
      \node[dot,label=right:$u''_i$] at (uiR) {};
      \node[dot,label=right:$u''_j$] at (ujR) {};
      \node[dot,label=right:$u''_k$] at (ukR) {};
    \end{scope}

    \def\matyshift{-1.2mm}
    \def\labelyshift{-1.2mm}

    \node (ZB) [anchor=north,font=\scriptsize]
      at ($(Bfig.south)+(0,\matyshift)$)
      {$\begingroup\setlength{\arraycolsep}{4pt}\displaystyle
        \hat Z=\begin{pmatrix}
          1 & 1 & 0 & 0 & 0 & 0 & 1\\
          0 & 0 & 1 & 1 & 0 & 0 & 0\\
          0 & 0 & 0 & 0 & 1 & 1 & 0\\
          1 & 0 & 1 & 0 & 0 & 1 & 0\\
          0 & 1 & 0 & 0 & 1 & 0 & 0\\
          0 & 0 & 0 & 1 & 0 & 0 & 1
        \end{pmatrix}\endgroup$};

    \node (ZS) [anchor=base,font=\scriptsize]
      at ($(Sfig.south |- ZB.base)$)
      {$\begingroup\setlength{\arraycolsep}{4pt}\displaystyle
        Z=\frac12\begin{pmatrix}
          2 & 1 & 0 & 1\\
          0 & 1 & 1 & 0\\
          0 & 0 & 1 & 1
        \end{pmatrix}\endgroup$};

    \node[font=\small] at ($(ZS.south |- ZB.south) + (0,\labelyshift)$) {(a)};
    \node[font=\small] at ($(ZB.south)          + (0,\labelyshift)$) {(b)};

  \end{tikzpicture}

  \caption{(a) Skeleton graph $S$. (b) The associated bipartite graph \(B\). The highlighted edges show the correspondence between \(S\) and \(B\). We also present the incidence matrices \(Z\) and \(\hat Z\).}
  \label{fig:StoB}
\end{figure}
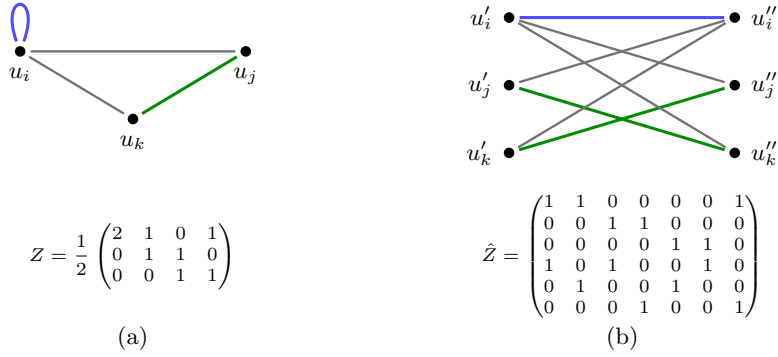

Let $\hat Z$ be the node-edge incidence matrix of $B$ (without the normalization factor $1/2$).  
We organize the rows of the matrix in a way that the $i$th (resp., $(i + q)$th) row, for $1\leq i \leq q$, corresponds to the node $u'_i$ (resp., $u''_i$). 
Given the vector $x\in \N^q$, we let $\hat x:= (x; x)\in \N^{2q}$.
We then introduce an auxiliary linear program:
\begin{equation}\label{eq:M-d}
\begin{aligned}
\max \; & \bfo^{\!\top} d \\[2pt]
\text{s.t. } & \hat Z d \le \hat x \quad \mbox{and} \quad d  \geq  0.
\end{aligned}
\end{equation}
Similarly, let $\cD_x$ be the set of solutions to~\eqref{eq:M-d}. 

Given a vector $c\in \cC_x$ (resp., $d\in \cD_x$), we let $c(u_i,u_j)$ (resp., $d(u'_i, u''_j)$) be the entry of $c$ (resp., $d$) corresponding to the edge $(u_i,u_j) \in E(S)$ (resp., $(u'_i,u''_j)\in E(B)$). In other words, we view $c$ (resp., $d$) as an element of $E^*(S)$ (resp., $E^*(B)$). 
We have the following result:

\begin{lemma}\label{lem:relcandd}
The following two items hold:
\begin{enumerate} 
\item The maximal value of~\eqref{eq:M-d} is the same as the maximal value of~\eqref{eq:P-c}. 
\item For any vertex $c\in \cC_x$, there exists a vertex $d\in \cD_x$ such that
$$
c(u_i, u_j) = 
\begin{cases}
d(u'_i, u''_i) & \mbox{if } u_i = u_j, \\
d(u'_i,u''_j) + d(u'_j, u''_i) & \mbox{if } u_i\neq u_j.
\end{cases}
$$
\end{enumerate}
\end{lemma}

\begin{proof}
To establish item~1, we first introduce two maps $\phi$ and $\psi$. The first map   
$\phi: E^*(B)\to E^*(S)$ is defined by sending a given $d$ to the element 
$c:= \phi(d)$ 
such that for any edge $(u_i,u_j)$ of $S$,  
$$
c(u_i,u_j):=
\begin{cases}
d(u'_i, u''_i) & \mbox{if } u_i = u_j, \\
d(u'_i, u''_j)+d(u'_j,u''_i) & \mbox{if } u_i \neq u_j.
\end{cases}
$$ 
The second map $\psi: E^*(S)\to E^*(B)$ is defined by sending a given~$c$ to the element~$d$ such that for any edge $(u'_i,u''_j)$ of $B$, 
$$
d(u'_i,u''_j):=
\begin{cases}
c(u_i,u_i) & \mbox{if } u_i = u_j, \\
\frac{1}{2}c(u_i,u_j) & \mbox{if } u_i\neq u_j. 
\end{cases}
$$
It is clear from the definition that $\bfo^\top \phi(d) = \bfo^\top d$ and $\bfo^\top \psi(c) = \bfo^\top c$. 
Also, note that $\phi\cdot \psi$ is the identity map. 

Now, let $\mu_c$ and $\mu_d$ be the maximal values of~\eqref{eq:P-c} and~\eqref{eq:M-d}, respectively. We first show that $\mu_c \geq \mu_d$.
Let $d$ be a solution to~\eqref{eq:M-d}, $(y',y'') := \hat Z d$, and $y:= Z \phi(d)$. Then, by the definition of $\phi$, it is not hard to see that $y = \frac{1}{2}(y' + y'')$ (one can show that the equality holds entry-wise) and hence, $y\leq x$.   
It follows that $\mu_c \geq \bfo^\top \phi(d) = \bfo^\top d = \mu_d$. 
Conversely, let $c$ be a solution to~\eqref{eq:P-c}, $y:= Z c$, and $(y',y''):= \hat Z\psi(c)$. Then, by the definition of $\psi$, we have that $y = y' = y''$, so $(y', y'') \leq \hat x$. It follows that $\mu_d \geq \bfo^\top \psi(c) = \bfo^\top c = \mu_c$.   We thus conclude that $\mu_c = \mu_d$.

To establish item~2, we note that $\psi(c)\in \cD_x$ for any $c\in \cC_x$ and, conversely, $\phi(d)\in \cC_x$ for any $d\in \cD_x$. Since $\phi\cdot \psi: \cC_x\to \cC_x$ is the identity map, $\phi: \cD_x\to \cC_x$ is surjective. Thus, for any vertex $c$ of $\cC_x$, there exists a vertex $d$ of $\cD_x$ such that $c = \phi(d)$.    
\end{proof}

Now, to prove Proposition~\ref{prop:ytoc}, it remains to establish the following lemma:

\begin{lemma} \label{lem:TU-D}
If $d$ is a vertex of $\cD_x$, then $d$ is integer valued. 
\end{lemma}

\begin{proof}
Consider the set
$Q := \{d\in\R^{|E(B)|}_{\ge  0} \mid \hat Z d \leq \hat x \}$. 
We claim that $Q$ is  an {\it integral} polytope. To wit, note that the matrix $\hat Z$, being the incidence matrix of a bipartite graph, is known~\cite[Thm.~5.25]{korteVygen2012} to be {\it totally unimodular} (i.e., the determinant of every minor of $\hat Z$ is either $0$, $1$, or $-1$). 
Since $\hat x$ is integer valued, 
by the Hoffman-Kruskal Theorem~\cite[Thm.~5.20]{korteVygen2012}, the polytope $Q$ is integral. Finally, note that the linear program~\eqref{eq:M-d} can be written as 
$$
\max \bfo^\top d \quad \mbox{s.t. } d\in Q.
$$
An extremal solution to the above linear program is necessarily a vertex of $Q$. This completes the proof.  
\end{proof}



\begin{remark}\normalfont
Integrality of~\eqref{eq:M-d} can also be seen from the fact that its dual is a {\it vertex cover} problem, whose extremal solutions are integer valued  for bipartite graphs~\cite{SchrijverCO}. 
\end{remark}

\subsection{Proof of Item~1 of Theorem~\ref{thm:main}}\label{sec:thm1item1}





If $y\in \cY_x$ is a vertex, then there exists a vertex $c\in \cC_x$ such that $y = Zc$. By Proposition~\ref{prop:ytoc}, any vertex of $\cC_x$ is integer valued. Thus, to establish item~1 of Theorem~\ref{thm:main}, it suffices to prove the following result: 

\begin{proposition}\label{prop:integrality}
    If $c$ is a vertex of $\cC_x$, then $y = Zc$ is integer valued.  
\end{proposition}

We establish below the proposition. 
A node $u_i\in S$ is said to be a {\it half-integer node} (resp. {\it integer node}) if $2y_i$ is an odd (resp. even) integer. 
We show below that there is no half-integer node, which will then establish item~1 of Theorem~\ref{thm:main}. The proof will be carried out by contradiction, i.e., we assume that there exist at least one half-integer node and derive a contradiction toward the end of the subsection. 

Let $S_c$ be the subgraph of $S$ induced by the support of~$c$. Specifically, we have that $V(S_c) = V(S)$ and 
\begin{equation*}\label{eq:ESc}
E(S_c) = \{(u_i,u_j)\in E(S) \mid c(u_i,u_j) > 0\}.
\end{equation*}
We have the following result:

\begin{lemma}\label{lem:evenhalfintnode}
Each connected component of $S_c$ contains an even number of half-integer nodes.
\end{lemma}

\begin{proof}
Let $S'$ be a connected component of $S_c$. It follows from the definition of $Z$ that 

\begin{equation}\label{eq:evenc}
\sum_{u_i\in S'} y_i
= \sum_{(u_i,u_j)\in E(S')} c(u_i,u_j)
\end{equation}
Since $c$ is integer-valued, the right hand side of~\eqref{eq:evenc} is an  integer, which implies that there exists an even number of half-integer nodes in $S'$. 
\end{proof}

By our hypothesis and by Lemma~\ref{lem:evenhalfintnode}, there exist at least $2$ distinct half-integer nodes $u_i$ and $u_j$, which belong to the same connected component of $S_c$. Now, let $p$ be a path from $u_i$ to $u_j$. We write $p$ explicitly as
\begin{equation}\label{eq:pathp}
p = u_{i_1}u_{i_2}\cdots u_{i_k}, \quad \mbox{with } u_{i_1} = u_i \mbox{ and  } u_{i_k} = u_j. 
\end{equation}
We can assume, without loss of generality, that all the nodes $u_{i_2},\ldots, u_{i_{k-1}}$ are integer nodes. 
We have 

\begin{lemma}\label{lem:patheven}
The length of $p$ is even.  
\end{lemma}

\begin{proof}
We assume to the contrary that the length of $p$ is odd (so $k$ is even).  
Since $c$ is integer valued and since $p$ is a path of $S_c$, $c(u_{i_\ell}, u_{i_{\ell+1}})\geq 1$ for any $\ell = 1,\ldots, k-1$.   
Define an element $c'\in \R^{m}_{\geq 0}$ as follows: 
For each edge $(u_{i_\ell}, u_{i_{\ell + 1}})\in E(p)$, let
\begin{equation}\label{eq:defc'}
c'(u_{i_\ell}, u_{i_{\ell + 1}}):= c(u_{i_\ell}, u_{i_{\ell + 1}}) + (-1)^{\ell-1}.
\end{equation}
For any other edge $(u_{i'},u_{j'})\notin E(p)$, we let $c'(u_{i'}, u_{j'}):= c(u_{i'},u_{j'})$. 
Now, let $y':= Z c'$. We claim that $y'\leq x$. To wit, 
we consider the following four cases: 
\begin{description}
 \item[{\it Case 1: $u_{i'}\notin V(p)$.}] In this case, $u_{i'}$ is not incident to any edge in the path $p$. Thus,  $y'_{i'} = y_{i'}$. 

\item[{\it Case 2: $u_{i'} = u_{i_\ell}$ for $\ell = 2,\ldots, k-1$.}] Since both $u_{i_{\ell- 1}}$ and $u_{i_{\ell + 1}}$ are neighbors of $u_{i_\ell}$, it follows from~\eqref{eq:defc'} that $y'_{i'} = y_{i'}$.  

\item[{\it Case 3: $u_{i'} = u_i$.}] Note that $(u_{i_1}, u_{i_2})$ is the only edge of $p$ incident to $u_i$. By\eqref{eq:defc'}, $c'(u_{i_1},u_{i_2}) = c(u_{i_1}, u_{i_2}) + 1$. It follows that  
$y'_i = y_i + \frac{1}{2}$.
Since $y_i$ is a half-integer and since $x_i$ is an integer, $y'_i\leq x_i$. 

\item[{\it Case 4: $u_{i'} = u_j$.}] Similarly, $(u_{i_{k-1}}, u_{i_k})$ is the only edge of $p$ incident to $u_{j}$. Since $k$ is even, it follows from~\eqref{eq:defc'} that
$c'(u_{i_{k-1}},u_{i_k})  = c(u_{i_{k-1}}, u_{i_k}) + 1$.
Thus, $y'_j = y_j + \frac{1}{2}\leq x_j$.  
\end{description}
We have thus {established} the claim. But then, $\bfo^\top y' = \bfo^\top y + 1$, contradicting the fact that $y$ is a solution to~\eqref{eq:LP-y}. 
\end{proof}

With the above lemma, we now prove Proposition~\ref{prop:integrality}.



\begin{proof}[Proof of Proposition~\ref{prop:integrality}]
Let $p$ be given as in~\eqref{eq:pathp}. 
Similar to what has been done in~\eqref{eq:defc'}, 
we introduce $c', c''\in \N_0^m$ as follows: For each edge $(u_{i'}, u_{j'})\in E(S) - E(p)$, we set $c'(u_{i'},u_{j'}) := c(u_{i'},u_{j'})$ and $c''(u_{i'},u_{j'}) := c(u_{i'},u_{j'})$. For each edge $(u_{i_\ell}, u_{i_{\ell + 1}})\in E(p)$, let 
\[
\begin{aligned}
c'(u_{i_\ell},u_{i_{\ell+1}}) & :=c(u_{i_\ell},u_{i_{\ell+1}})+(-1)^{\ell-1},\\
c''(u_{i_\ell},u_{i_{\ell+1}}) & :=c(u_{i_\ell},u_{i_{\ell+1}})-(-1)^{\ell-1}.
\end{aligned}
\] 
Note that $c(u_{i_\ell}, u_{i_{\ell + 1}}) \geq 1$ for all $\ell = 1,\dots, k-1$, so $c',c''$ indeed belong to $\N_0^m$. 
It is clear that 
$\bfo^\top c' = \bfo^\top c'' = \bfo^\top c$.  
We further let 
$y':=Zc'$ and $y'':=Zc''$. 
Using the same arguments in the proof of Lemma~\ref{lem:patheven}, we have that if $u_{k}\notin\{u_i, u_j\}$, then $y'_k = y''_k = y_k$. 
For node $u_i$ and $u_j$, we have 
\[
\begin{aligned}
y'_i=y_i+\tfrac12,\qquad & y'_j= y_j + \tfrac12(-1)^{k-2} =  y_j-\tfrac12,\\
y''_i=y_i-\tfrac12, \qquad & y''_j=y_j - \tfrac12(-1)^{k-2} = y_j + \tfrac12,
\end{aligned}
\]
where the fact $(-1)^{k-2} = -1$ follows from Lemma~\ref{lem:patheven} (note that the length of $p$ is $(k-1)$).
Since $u_i$ and $u_j$ are half-integer nodes, $y'\leq x$ and $y'' \leq x$. 
The above arguments imply that both $c'$ and $c''$ are solutions to~\eqref{eq:P-c} and, moreover, $c = \frac{1}{2}(c' + c'')$, contradicting the hypothesis that $c$ is a vertex of $\cC_x$. We thus conclude that the path $p$ does not exist, so $y = Zc$ is integer valued.  
\end{proof}

\subsection{Proof of Item~2 of Theorem~\ref{thm:main}}
In this subsection, we show that an integer solution~$y$ to~\eqref{eq:LP-y} can be found in $O(q^3)$.  
This is done by first constructing the bipartite graph $B$ introduced in Subsection~\ref{sec:half-int}, next obtaining an integer solution $d$ to~\eqref{eq:M-d}, and then setting $c:=\phi(d)$ where $\phi$ is introduced in the proof of Lemma~\ref{lem:relcandd}, and setting $y := Zc$. 
If $y$ has half-integer entries (i.e., $y_i - 1/2\in \N_0$), then the arguments in the proofs of Lemmas~\ref{lem:evenhalfintnode} and~\ref{lem:patheven} show that in $S_c$ one can find a (shortest) path of even length joining two half-integer nodes in the same connected component. 
We then modify $c$ along this path in the same alternating way as in the proof of Proposition~\ref{prop:integrality}, and update $y:=Zc$. 
This reduces the number of half-integer entries of $y$ by $2$. 
Iterating at most $\lfloor q/2 \rfloor$ times yields an integer solution~$y$ to~\eqref{eq:LP-y}.
We elaborate below on the time complexity:
\begin{description}[leftmargin=1.5em]
\item[\it Step 1.]  
By its definition, $B$ has $2q$ nodes and less than $2m$ edges and $\hat x$ has $2q$ entries, so construction of these two objects can be done in $O(q + m)$. 
The linear program~\eqref{eq:M-d} is an unweighted bipartite $b$-matching problem, which can be reformulated, in $O(q+m)$, as a maximum flow problem~\cite[Sec.~26.3]{CLRS2009}, 
whose solution can be obtained in $O(q^3)$. 

\item[\it Step 2.] It follows from the definitions of $\phi$ and $Z$ that the computation $\phi(d)$ can be done in $O(m)$, and computation of $y = Zc$ can be done in $O(m)$.  


\item[\it Step 3.] If $y$ has a half-integer node $u_i$, then there exists another half-integer node $u_j$, linked by a path of even length. Such a node  $u_j$ can be obtained in $O(q + m)$. Update $c$ and $y$ as described in the proof of Proposition~\ref{prop:integrality}. Iterating at most $\lfloor q/2 \rfloor$ times yields an integer solution~$y$. 
\end{description}
We summarize the above steps in Algorithm~\ref{alg:item1}.

\begin{algorithm}[ht]
  \caption{}
  \label{alg:item1}
  \textbf{Input:}
  A graph $S$ on $q$ nodes and $m$ edges (possibly with self-loops),
  and an integer vector $x\in\N^q$.\\[2pt]
  \textbf{Output:}
  An integer solution $y$ to~\eqref{eq:LP-y}.

  \begin{enumerate}[leftmargin=1.5em,label=\arabic*.]
    \item[\it 1.] Construct $B$ and $\hat x:=(x;x)$; \\ Solve the linear program~\eqref{eq:M-d} and obtain an integer solution $d$. 
    \hfill{$O(q^3)$}
    \item[\it 2.] Set $c:= \phi(d)$ and $y:=Z c$. 
           \hfill{$O(m)$}
    \item[\it 3.] Iterating to eliminate half-integer nodes. \hfill{$O(q^2+qm)$}
  \end{enumerate}
\end{algorithm}



\section{Hamiltonicity of Complete $S$-partite Graphs}\label{sec:KyHam}
In this section, we establish the following theorem:

\begin{theorem}\label{thm:hc}
    Let $S$ be an undirected, connected graph, possibly with self-loops. 
    Let $c\in \N^m$ and $y\in \N^q$ be such that $y = Zc$ and  $\bfo^\top c  = \bfo^\top y \geq 3$. 
    Then, the complete $S$-partite graph $K_y$ has a Hamilton cycle.   
\end{theorem}

Hamiltonicity of $S$-partite graphs, for $S$ a directed graph, has been addressed in~\cite{chen2025hamiltonicity}. We develop below a new approach to establish Theorem~\ref{thm:hc}. 

Consider the auxiliary undirected pseudograph $M$ whose node set $V(M)$ is the same as that of $S$, and whose edge set $E(M)$ is defined as follows: For each  edge $(u_i,u_j)\in E(S)$ (possibly with $u_i= u_j$), we place  $c(u_i,u_j)$ multiple edges between $u_i$ and $u_j$ (self-loops if $u_i = u_j$). 
It follows directly  from the construction that
\begin{equation}\label{eq:degM}
\deg_M(u_i)=\sum_{j\ne i} c(u_i,u_j)+2c(u_i,u_i)=2y_i > 0,
\end{equation}
where $c(u_i,u_j) = 0$ if $(u_i,u_j)\notin E(S)$. In particular, every node of $M$ has an even degree.  

Since $S$ is connected and since $c(u_i,u_j) > 0$ for all $(u_i,u_j)\in E(S)$, we have that $M$ is connected. 
Then, by~\eqref{eq:degM} and Hierholzer's theorem~\cite{Hierholzer1873},  there exists an Euler circuit of $M$, i.e., a closed walk that traverses each edge of $M$ exactly once. We write this circuit explicitly as
\begin{equation}\label{eq:eulercircuit}
  \cE:=
  u_{i_0} e_1  u_{i_1} e_2, \dots, e_L u_{i_L}, 
\end{equation}
where the nodes and the edges are alternating with the edges $e_j$ linking the nodes $u_{i_{j-1}}$ and $u_{i_j}$, $u_{i_L} = u_{i_0}$, and $L:=|E(M)|$ is the size of $M$.

We illustrate $M$ and $\mathcal{E}$ in the following example: 

\begin{example}\label{exmp:1}\normalfont
Let $S$ be the $3$-cycle, with $V(S) = \{u_1,u_2,u_3\}$. Let $c(u_1,u_2)=4$, $c(u_2,u_3)=2$, and $c(u_3,u_1)=2$. Then $y=Zc=(3,3,2)$. Figure~\ref{fig:M-Ky-hamilton}(a) shows the corresponding Eulerian pseudograph $M$ together with the chosen Euler circuit 
$\mathcal{E}=u_1 e_1 u_2 e_2 u_3 e_3 u_1 e_4 u_2 e_5 u_3 e_6 u_1 e_7 u_2 e_8 u_1$.
\end{example}

\begin{figure}[ht]
  \centering
  \begin{tikzpicture}[
      line cap=round, line join=round,
      dot/.style={circle,inner sep=0pt,minimum size=4pt,fill=black},
      partbox/.style={draw,rounded corners,inner sep=2pt},
      edgegap/.style={shorten >=2pt,shorten <=2pt},
      euleredge/.style={thick,blue,edgegap},
      kyedge/.style={gray!40,line width=0.4pt,edgegap},
      cycleedge/.style={thick,blue,edgegap},
      edgelabel/.style={font=\scriptsize,fill=white,inner sep=1pt},
      every label/.style={font=\small}
    ]

    \begin{scope}[scale=1.05]
      \coordinate (m1) at (0,1.4);
      \coordinate (m2) at (-2.0,-0.6);
      \coordinate (m3) at (2.0,-0.6);

      \node[dot,label=above:$u_1$]       (u1) at (m1) {};
      \node[dot,label=below left:$u_2$]  (u2) at (m2) {};
      \node[dot,label=below right:$u_3$] (u3) at (m3) {};

      \draw[euleredge] (u1) to[bend left=44]
        node[edgelabel,pos=.30] {$e_7$} (u2);
      \draw[euleredge] (u1) to[bend right=6]
        node[edgelabel,pos=.40] {$e_4$} (u2);
      \draw[euleredge] (u1) to[bend right=28]
        node[edgelabel,pos=.60] {$e_1$} (u2);
      \draw[euleredge] (u2) to[bend right=16]
        node[edgelabel,pos=.40] {$e_8$} (u1);

      \draw[euleredge] (u2) to[bend left=18]
        node[edgelabel,pos=.55] {$e_2$} (u3);
      \draw[euleredge] (u2) to[bend right=18]
        node[edgelabel,pos=.55] {$e_5$} (u3);

      \draw[euleredge] (u3) to[bend left=8]
        node[edgelabel,pos=.55] {$e_6$} (u1);
      \draw[euleredge] (u3) to[bend right=20]
        node[edgelabel,pos=.52] {$e_3$} (u1);

      \node[font=\small] at (0,-1.6)
        {(a) };
    \end{scope}

    \begin{scope}[xshift=6.25cm,scale=1.0, every label/.style={font=\scriptsize}]

      \node[dot,label={[xshift=-3pt]above:$\tau(0)$}]  (v11) at (-0.3,1.7) {};
      \node[dot,label={[xshift=7pt]above:$\tau(3)$}]       (v12) at (0.3,1.15) {};
      \node[dot,label={[xshift=3pt]above right:$\tau(6)$}]        (v13) at (0.9,0.6)  {};

      \node[dot,label=left:$\tau(1)$]        (v21) at (-2.0,0.3)  {};
      \node[dot,label={[xshift=-10pt]left:$\tau(4)$}]        (v22) at (-1.65,-0.3) {};
      \node[dot,label={[xshift=-20pt]left:$\tau(7)$}]        (v23) at (-1.3,-0.9) {};

      \node[dot,label=right:$\tau(2)$]       (v31) at (2.4,0)  {};
      \node[dot,label=right:$\tau(5)$]       (v32) at (2.4,-0.9) {};

      \node[partbox,fit=(v11)(v12)(v13)] (Box1) {};
      \node[partbox,fit=(v21)(v22)(v23)] (Box2) {};
      \node[partbox,fit=(v31)(v32)]      (Box3) {};
      
      \node[font=\scriptsize,anchor=south,xshift=10pt,yshift=1pt] at (Box1.north)
        {$V_1$};
      \node[font=\scriptsize,anchor=north,yshift=-1pt] at (Box2.south)
        {$V_2$};
      \node[font=\scriptsize,anchor=north,yshift=-1pt] at (Box3.south)
        {$V_3$};

      \foreach \a in {11,12,13}{
        \foreach \b in {21,22,23}{
          \draw[kyedge] (v\a) -- (v\b);
        }
      }

      \foreach \a in {21,22,23}{
        \foreach \b in {31,32}{
          \draw[kyedge] (v\a) -- (v\b);
        }
      }

      \draw[kyedge] (v11) to[bend left=24] (v31);
      \draw[kyedge] (v11) to[bend left=32] (v32);
      \draw[kyedge] (v12) to[bend left=20] (v31);
      \draw[kyedge] (v12) to[bend left=28] (v32);
      \draw[kyedge] (v13) to[bend left=16] (v31);
      \draw[kyedge] (v13) to[bend left=24] (v32);

      \draw[cycleedge] (v11) -- (v21);
      \draw[cycleedge] (v21) -- (v31);
      \draw[cycleedge] (v31) -- (v12);
      \draw[cycleedge] (v12) -- (v22);
      \draw[cycleedge] (v22) -- (v32);
      \draw[cycleedge] (v32) -- (v13);
      \draw[cycleedge] (v13) -- (v23);
      \draw[cycleedge] (v23) -- (v11);

      \node[font=\small] at (0,-1.6)
        {(b)};
    \end{scope}

  \end{tikzpicture}
  
  \par\vspace{-0.85\baselineskip}
  \caption{(a) The auxiliary pseudograph $M$ for Example~\ref{exmp:1}, with blue edges labeled by their order in $\mathcal{E}$. (b) The corresponding complete $S$-partite graph $K_y$, with the lifted Hamilton cycle $H$ shown in blue.}
  \label{fig:M-Ky-hamilton}
\end{figure}
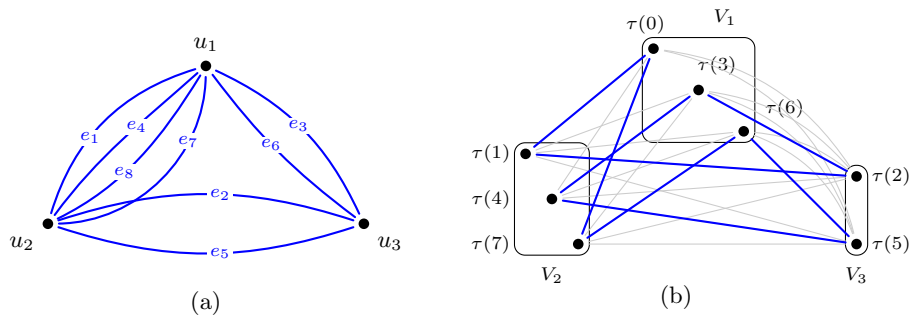

Next, we make the following observation: 

\begin{lemma}\label{lem:closed-walk-lift}
Let $u_{i_0},\ldots, u_{i_{L-1}}$ be the sequence of nodes in the Euler circuit~\eqref{eq:eulercircuit}. 
Then, for any sequence of nodes $v_{i_0},\ldots, v_{i_{L-1}}$ such that $v_{i_j}\in \pi^{-1}(u_{i_j})$  and $v_{i_j} \neq v_{i_{j + 1}}$ for all $j = 0,\ldots, L-1$ (with $v_{i_L}$ identified with $v_{i_0}$),  
$v_{i_0}v_{i_1}\cdots v_{i_{L-1}}v_{i_0}$ is a closed walk of $K_y$.
\end{lemma}

\begin{proof}
We need to show that for each $j = 1,\ldots, L$, $v_{i_{j-1}}v_{i_{j}}$ is an edge of $K_y$ (where $v_{i_L}$ is identified with $v_{i_0}$). 
Since $K_y$ is complete $S$-partite and since $\pi(v_{i_j}) = u_{i_j}$ for all $j$, it suffices to show that $(u_{i_{j-1}}, u_{i_j})$ is an edge of $S$. But this holds because $e_{j}$ is an edge of $M$ linking $u_{i_{j-1}}$ and $u_{i_j}$. By construction of $M$, $e_j$ exists if and only if $c(u_i,u_j) > 0$.    
\end{proof}

With the preliminaries above, we establish Theorem~\ref{thm:hc}: 

\begin{proof}[Proof of Theorem~\ref{thm:hc}]
We will make a specific choice of the $v_{i_j}$'s in the statement of Lemma~\ref{lem:closed-walk-lift} to construct a desired Hamilton cycle of $K_y$. For each $p = 1,\ldots, q$, we let
$$\mathcal{I}_p:=\{j\in \{0,\ldots, L-1\} \mid u_{i_{j}} = u_p \}.$$
Since $\mathcal{E}$ is an Euler circuit, we have that 
$2|\mathcal{I}_p| = \deg_M(u_p)$. Then, by~\eqref{eq:degM}, 
$|\mathcal{I}_p| = y_p$. 
Thus there exists a bijection between $\mathcal{I}_p$ and $\pi^{-1}(u_p)$ for each $p = 1,\ldots, q$, which we denote by $\tau_p: \mathcal{I}_p \to \pi^{-1}(u_p)$. 
Since $\mathcal{I}_1,\ldots, \mathcal{I}_q$ form a partition of $\{0,\ldots, L-1\}$ and since $\pi^{-1}(u_1), \ldots, \pi^{-1}(u_q)$ form a partition of $V(K_y)$, we have the bijection $$\tau: \{0,\ldots, L-1\} \to V(K_y)$$ 
given by
$\tau(j):= \tau_p(j)$, where $p$ is such that $j\in \mathcal{I}_p$. Let 
$$
H:= \tau(0)\tau(1) \cdots \tau(L-1)\tau(0). 
$$
By construction and by Lemma~\ref{lem:closed-walk-lift}, $H$ is a closed walk of $K_y$ that visits every node exactly once. Finally, since $\bfo^\top c = \bfo^\top y \geq 3$, $K_y$ has at least $3$ nodes. We conclude that $H$ is a Hamilton cycle of $K_y$. 
\end{proof}

Figure~\ref{fig:M-Ky-hamilton}(b) illustrates how the Euler circuit in panel~(a) is lifted to the Hamilton cycle $H$ in $K_y$. For the example,
$\mathcal{I}_1=\{0,3,6\},\;\mathcal{I}_2=\{1,4,7\},\;\mathcal{I}_3=\{2,5\}.$
Accordingly, the three parts in panel~(b) contain the vertex sets $\{\tau(0),\tau(3),\tau(6)\}$, $\{\tau(1),\tau(4),\tau(7)\}$, and $\{\tau(2),\tau(5)\}$, respectively. Moreover, in this example, $e_1$ lifts to $(\tau(0),\tau(1))$, $e_2$ lifts to $(\tau(1),\tau(2))$, \ldots, and $e_8$ lifts to $(\tau(7),\tau(0))$. Hence the blue cycle is exactly $H=\tau(0)\tau(1)\cdots\tau(7)\tau(0).$

\section{Proof of Item~3 of Theorem~\ref{thm:main}}\label{sec:thm1item3}
We first note that if $K_y$ has a cycle cover, then $y\in \cX$ (similar arguments in~\cite{belabbas2021h} can be used to establish this fact). In particular, if $K_y$ is a largest subgraph of $K_x$, with $y\leq x$, which has a cycle cover, then the order of $K_y$ is necessarily bounded above by the maximal value of the linear program~\eqref{eq:LP-y}.




We show below that if $y$ is an integer solution to~\eqref{eq:LP-y}, then $K_y$ has a cycle cover with at most $q$ cycles.  
Let $c\in \N_0^m$ be such that $y = Zc$, 
and $S_c$ be the subgraph of $S$ induced by the support of $c$ (the definition is given in Subsection~\ref{sec:thm1item1}). A connected component $S'$ of $S_c$ is said to be {\it nontrivial} if it contains at least one edge. 
We denote by $\mathcal{S}$ the set of nontrivial connected components of $S_c$. 

For each $S'\in \mathcal{S}$ and for each edge $f = (u_i,u_j)\in E(S')$, we let $z'(f) := \frac{1}{2}(e_i + e_j) \in \R^{|V(S')|}$ be the corresponding incidence vector. We then define 
\begin{equation*}\label{eq:yS'}
y_{S'} := \sum_{f\in E(S')} c(f) z'(f) > 0. 
\end{equation*}
Let $K_{y_{S'}}$ be the complete $S'$-partite graph, viewed as a subgraph of $K_{y}$. We have the following result: 

\begin{lemma}\label{lem:no-small-components}
For any $S'\in \mathcal{S}$, $K_{y_{S'}}$ is connected and has at least three vertices. 
\end{lemma}

\begin{proof}
The connectedness of $K_{y_{S'}}$ directly follows from the fact that $S'$ is connected and $y_{S'} > 0$. 
We show below that $K_{y_{S'}}$ has at least three nodes. 
Suppose to the contrary that $K_{y_{S'}}$ has at most two nodes; then, by the fact that every entry of $y_{S'}$ is a positive integer, we must have that $S'$ has at most two nodes. We consider two cases:

\xc{Case 1: $S'$ has only one single node $u_i$.} 
Because $S'$ is a nontrivial connected component of $S_c$, the node $u_i$ must have a self-loop and, moreover, 
$y_{S'} = y_i = c(u_i,u_i) \leq 2$.  
Now, let $y' := y + e_i$. 
We claim that $y'$ satisfies the constraint in the linear program~\eqref{eq:LP-y}, i.e., $y'\in \cX$  and $y'\leq x$. 
This holds because $y\in \cX$  with $y_i \leq 2 < 3 \leq x_i$ and because $e_i$ is a column vector of $Z$ corresponding to the self-loop on $u_i$. 
But then, $\bfo^\top y' = \bfo^\top y + 1$, which contradicts the fact $y$, being a solution to~\eqref{eq:LP-y}, maximizes the objective function.

\xc{Case 2: $S'$ has two distinct nodes $u_i$ and $u_j$.} In this case, $S'$ has an edge $(u_i, u_j)$ with $c(u_i,u_j) > 0$, and necessarily
\begin{equation}\label{eq:case2ys'}
y_{S'} = 
\begin{bmatrix}
y_i \\
y_j
\end{bmatrix} = 
\begin{bmatrix}
1 \\ 
1
\end{bmatrix}.
\end{equation} 
Furthermore, we must have that $(u_i, u_j)$ is the only edge incident to either $u_i$ or $u_j$ such that $c(u_i,u_j) > 0$. 
To see this, note that if $S$ has the self-loop $(u_i, u_i)$, then $c(u_i,u_i) = 0$ because otherwise, $y_i = c(u_i,u_i) + \frac{1}{2}c(u_i,u_j) > 1$ contradicting~\eqref{eq:case2ys'}. The same arguments apply for $c(u_j,u_j)$. But then, by setting
$y'':= y + e_i + e_j$, 
we can similarly conclude that $y''\in \cX$ and $y''\leq x$, thus contradicting the fact that $y$ is a solution to~\eqref{eq:LP-y}. 
\end{proof}
We now appeal to Theorem~\ref{thm:hc} to obtain a Hamilton cycle $H_{S'}$ of $K_{y_{S'}}$. 
Since the subgraphs in $\mathcal{S}$ are pairwise disjoint and since $\sum_{S'\in \mathcal{S}} \|y_{S'}\|_1 = \|y\|$, we have that $H:= \cup_{S'\in \mathcal{S}} H_{S'}$ is a cycle cover of $K_y$, which has $|\mathcal{S}|$ cycles with $|\mathcal{S}|\leq |V(S)| =  q$. \hfill{\qed} 

Figure~\ref{fig:Sc-M-Ky} illustrates the above construction on a small instance. In this example, the skeleton in panel~(a) has edge set $\{(u_1,u_1),(u_1,u_2),(u_2,u_3)\}$ and capacity vector $x=(3,3,3)$. An optimal coefficient vector satisfies $c(u_1,u_1)=3$ and $c(u_2,u_3)=6$, with all other coefficients equal to zero, so the support graph $S_c$ shown in panel~(b) has two nontrivial connected components. Panels~(c) and~(d) summarize the construction of Section~\ref{sec:KyHam} applied separately to these two components. The loop component at $u_1$ yields a $3$-cycle on $V_1$, while the component on $\{u_2,u_3\}$ yields a $6$-cycle on $V_2\cup V_3$, together they form the desired cycle cover $H\subseteq K_y$.

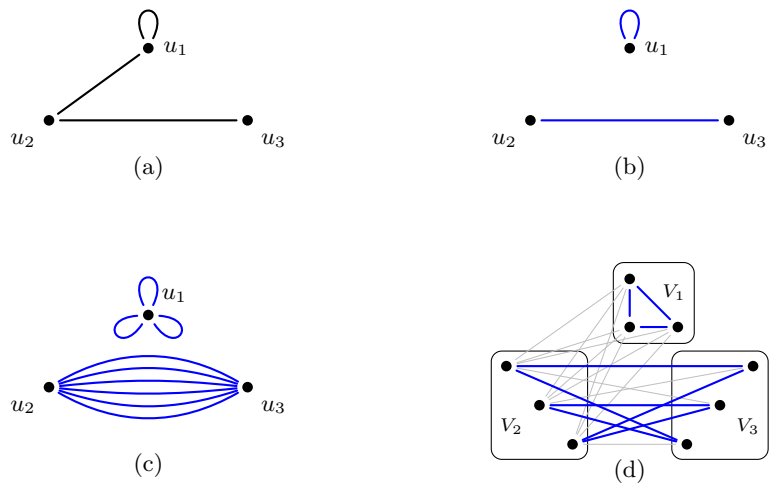
\begin{figure}[ht]
  \centering
  \begin{tikzpicture}[
      line cap=round, line join=round, >=Latex,
      dot/.style={circle,inner sep=0pt,minimum size=4pt,fill=black},
      partbox/.style={draw,rounded corners,inner sep=2pt},
      edgegap/.style={shorten >=2pt,shorten <=2pt},
      skedge/.style={thick,edgegap},
      skloop/.style={thick,edgegap},
      supportedge/.style={thick,blue,edgegap},
      supportloop/.style={thick,blue,edgegap},
      kyedge/.style={gray!50,line width=0.35pt,edgegap},
      cycleedge/.style={thick,blue,edgegap},
      every label/.style={font=\small}
    ]

    \begin{scope}[scale=.8]
      \coordinate (a1) at (0,1.05);
      \coordinate (a2) at (-1.65,-0.15);
      \coordinate (a3) at (1.65,-0.15);

      \node[dot,label=right:$u_1$]       (ux1) at (a1) {};
      \node[dot,label=below left:$u_2$]  (ux2) at (a2) {};
      \node[dot,label=below right:$u_3$] (ux3) at (a3) {};

      \draw[skedge] (ux1) -- (ux2);
      \draw[skedge] (ux2) -- (ux3);

      \draw[skloop] (ux1)
        .. controls +(-0.38,0.78) and +(0.38,0.78) .. (ux1);

      \node[font=\small] at (0,-0.95) {(a)};
    \end{scope}

    \begin{scope}[xshift=6.4cm,scale=.8]
      \coordinate (b1) at (0,1.05);
      \coordinate (b2) at (-1.65,-0.15);
      \coordinate (b3) at (1.65,-0.15);

      \node[dot,label=right:$u_1$]       (uc1) at (b1) {};
      \node[dot,label=below left:$u_2$]  (uc2) at (b2) {};
      \node[dot,label=below right:$u_3$] (uc3) at (b3) {};

      \draw[supportloop] (uc1)
        .. controls +(-0.38,0.78) and +(0.38,0.78) .. (uc1);
      \draw[supportedge] (uc2) -- (uc3);

      \node[font=\small] at (0,-0.95) {(b)};
    \end{scope}

    \begin{scope}[yshift=-3.55cm,scale=.8]
      \coordinate (m1) at (0,1.05);
      \coordinate (m2) at (-1.65,-0.15);
      \coordinate (m3) at (1.65,-0.15);

      \node[dot,label=above right:$u_1$]  (w1) at (m1) {};
      \node[dot,label=below left:$u_2$]   (w2) at (m2) {};
      \node[dot,label=below right:$u_3$]  (w3) at (m3) {};

      \draw[supportloop] (w1)
        .. controls +(-0.38,0.78) and +(0.38,0.78) .. (w1);
      \draw[supportloop,rotate around={120:(w1)}] (w1)
        .. controls +(-0.38,0.78) and +(0.38,0.78) .. (w1);
      \draw[supportloop,rotate around={240:(w1)}] (w1)
        .. controls +(-0.38,0.78) and +(0.38,0.78) .. (w1);

      \draw[supportedge,bend left=30]  (w2) to (w3);
      \draw[supportedge,bend left=18]  (w2) to (w3);
      \draw[supportedge,bend left=6]   (w2) to (w3);
      \draw[supportedge,bend right=6]  (w2) to (w3);
      \draw[supportedge,bend right=18] (w2) to (w3);
      \draw[supportedge,bend right=30] (w2) to (w3);

      \node[font=\small] at (0,-1.45) {(c)};
    \end{scope}

    \begin{scope}[xshift=6.4cm,yshift=-3.55cm,scale=.8]

      \node[dot] (v11) at (0.00,1.65) {};
      \node[dot] (v12) at (0.00,0.85) {};
      \node[dot] (v13) at (0.80,0.85) {};

      \node[dot] (v21) at (-2.05,0.20) {};
      \node[dot] (v22) at (-1.50,-0.45) {};
      \node[dot] (v23) at (-0.95,-1.10) {};

      \node[dot] (v31) at (0.95,-1.10) {};
      \node[dot] (v32) at (1.50,-0.45) {};
      \node[dot] (v33) at (2.05,0.20) {};

      \node[partbox,inner sep=4pt,fit=(v11)(v12)(v13)] (Box1) {};
      \node[partbox,inner sep=3.5pt,fit=(v21)(v22)(v23)] (Box2) {};
      \node[partbox,inner sep=3.5pt,fit=(v31)(v32)(v33)] (Box3) {};

      \node[font=\scriptsize,anchor=east,yshift=4pt] at (Box1.east) {$V_1$};
      \node[font=\scriptsize,anchor=west,yshift=-8pt] at (Box2.west) {$V_2$};
      \node[font=\scriptsize,anchor=east,yshift=-8pt] at (Box3.east) {$V_3$};

      \foreach \a in {11,12,13}{
        \foreach \b in {21,22,23}{
          \draw[kyedge] (v\a) -- (v\b);
        }
      }

      \foreach \a in {21,22,23}{
        \foreach \b in {31,32,33}{
          \draw[kyedge] (v\a) -- (v\b);
        }
      }

      \foreach \a/\b in {11/12,11/13,12/13}{
        \draw[kyedge] (v\a) -- (v\b);
      }

      \draw[cycleedge] (v11) -- (v12);
      \draw[cycleedge] (v12) -- (v13);
      \draw[cycleedge] (v13) -- (v11);

      \draw[cycleedge] (v21) -- (v31);
      \draw[cycleedge] (v31) -- (v22);
      \draw[cycleedge] (v22) -- (v32);
      \draw[cycleedge] (v32) -- (v23);
      \draw[cycleedge] (v23) -- (v33);
      \draw[cycleedge] (v33) -- (v21);

      \node[font=\small] at (0,-1.55) {(d)};
    \end{scope}

  \end{tikzpicture}

  \par\vspace{-0.85\baselineskip}
  \caption{An instance in which the support graph $S_c$ has two nontrivial connected components. 
  (a) The skeleton $S$.
  (b) The support graph $S_c$.
  (c) The associated pseudograph $M_c$.
  (d) The blow-up $K_y$ and the resulting cycle cover $H\subseteq K_y$.}
  \label{fig:Sc-M-Ky}
\end{figure}

\section{Numerical Study}\label{sec:ns}

We fix the skeleton graph $S$ shown in Figure~\ref{fig:sim-skeleton}, obtained from the $6$-cycle on $\{u_1,\dots,u_6\}$ by adding the chords $(u_1,u_3)$ and $(u_3,u_5)$ and a self-loop at $u_2$. For each vector $x\in\mathbb{N}^6$, let $K_x$ be the associated complete $S$-partite graph and let $n:=\|x\|_1$. Solving~\eqref{eq:LP-y} on $K_x$ by the max-flow algorithm developed above yields an integer solution $y^*$. We then set $n^*=\|y^*\|_1$.

\begin{figure}[ht]
  \centering
  \begin{tikzpicture}[
      line cap=round, line join=round, >=Latex,
      dot/.style={circle,inner sep=0pt,minimum size=4pt,fill=black},
      edgegap/.style={shorten >=2pt,shorten <=2pt},
      skedge/.style={draw=black!55,line width=1.0pt,edgegap},
      skloop/.style={draw=black!55,line width=1.0pt,edgegap},
      every label/.style={font=\small}
    ]

    \begin{scope}[scale=1]
      \coordinate (a1) at ( 2.05,  0.00);
      \coordinate (a2) at ( 0.95,  0.62);
      \coordinate (a3) at (-0.40,  0.62);
      \coordinate (a4) at (-1.55,  0.00);
      \coordinate (a5) at (-0.40, -0.62);
      \coordinate (a6) at ( 0.95, -0.62);

      \node[dot,label=right:$u_1$]       (u1) at (a1) {};
      \node[dot,label=above right:$u_2$] (u2) at (a2) {};
      \node[dot,label=above left:$u_3$]  (u3) at (a3) {};
      \node[dot,label=left:$u_4$]        (u4) at (a4) {};
      \node[dot,label=below left:$u_5$]  (u5) at (a5) {};
      \node[dot,label=below right:$u_6$] (u6) at (a6) {};

      \draw[skedge] (u1) -- (u2);
      \draw[skedge] (u2) -- (u3);
      \draw[skedge] (u3) -- (u4);
      \draw[skedge] (u4) -- (u5);
      \draw[skedge] (u5) -- (u6);
      \draw[skedge] (u6) -- (u1);
      \draw[skedge] (u1) -- (u3);
      \draw[skedge] (u3) -- (u5);

      \draw[skloop] (u2)
       .. controls +(-0.24,0.72) and +(0.24,0.72) .. (u2);
    \end{scope}
  \end{tikzpicture}
  \caption{The skeleton graph $S$ considered in the numerical study.}
  \label{fig:sim-skeleton}
\end{figure}
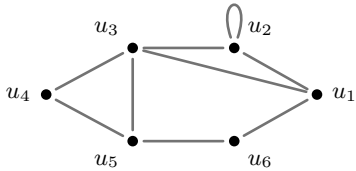

For each chosen $x$, we sample a random $S$-partite graph $G$ by retaining each admissible edge of $K_x$ independently with probability $p$. Equivalently, $G$ is drawn from a stochastic block model~\cite{holland1983stochastic} with community sizes $x_i$ and edge-probability matrix given by $p$ on the support of $S$ and $0$ elsewhere. For each sampled graph $G$, we define
\[
n(G):=\max\{|V(H)| \mid H\subseteq G \text{ is $2$-regular}\},
\]
that is, the maximum value of~\eqref{eq:mainproblem} with $K_x$ replaced by $G$.

To compute $n(G)$, we first repeatedly delete all vertices of degree less than $2$ until no such vertex remains (we can do so because none of these vertices can belong to a $2$-regular subgraph). We then decompose the resulting graph into connected components and, for each component $C$, solve the following integer linear program. The maximum values of all connected components sum to $n(G)$:
\begin{equation*}\label{eq:sim-ilp}
\begin{aligned}
\max\;& \sum_{v\in V(C)} s_v\\
\text{s.t.}\;& \sum_{e\in \delta_C(v)} h_e = 2s_v,\qquad &&\forall v\in V(C),\\
& h_e \le s_u,\qquad h_e \le s_v,\qquad &&\forall e=(u,v)\in E(C),\\
& s_v\in\{0,1\},\qquad &&\forall v\in V(C),\\
& h_e\in\{0,1\},\qquad &&\forall e\in E(C),
\end{aligned}
\end{equation*}
where $\delta_C(v)$ denotes the set of edges of $C$ incident to $v$.

We sample $N$ independent graphs $G^{(1)},\dots,G^{(N)}$,  
and plot the empirical probability mass function (PMF):
\[
\widehat{\mathbb{P}}(n(G)=t)
:=\frac{1}{N}\sum_{r=1}^N \mathds{1}_{\{n(G^{(r)})=t\}},
\quad t=0,1,\dots,n^*,
\]
where $\mathds{1}_{A}$ is the indicator function. 
We are particularly interested in the value 
$$\hat p^* := \widehat{\mathbb{P}}(n(G)=n^*).$$

We start with the numerical study for a relatively small $n$. 
Specifically, we set $x:=(24,7,4,11,6,4)$ and, correspondingly, $n=56$.
We consider different values of $p$: 
\[
p\in\left\{n^{-0.5},\,n^{-0.4},\,\tfrac{\log n}{n},\,\tfrac{4\log n}{n},\,\tfrac{6\log n}{n},\,0.6\right\},
\]
with $\log n/n \approx 0.072$, $n^{-0.5}\approx 0.134$, and $n^{-0.4}\approx 0.200$. 
For each $p$, we generate $N = 10^5$ samples. The corresponding empirical PMFs are given in 
Figure~\ref{fig:sim1}. 
Panel~(a) plots the entire PMFs, while panel~(b) zooms in near $n^*$. The results show that $\hat p^*$ is negligible for $p=\log n/n$ and $p=n^{-0.5}$, small for $p=n^{-0.4}$, substantial for $p=4\log n/n$, dominant for $p=6\log n/n$, and nearly $1$ for $p=0.6$.

\input{fig5}

We next consider a mid-size graph, with $x=(120,35,22,55,30,18)$ and $n=280$, and let 
\[
p\in\left\{n^{-0.4},\,\tfrac{4\log n}{n},\,\tfrac{6\log n}{n},\,0.6\right\},
\]
with $\log n/n \approx 0.020$ and $n^{-0.4} \approx 0.105$. For each $p$, we again generate $N = 10^5$ samples. The corresponding empirical PMFs are given in
Figure~\ref{fig:sim2}.  
Since the PMFs are concentrated near $n^*$, the plot only shows the relevant region.

\input{fig6}

Finally, we plot in Figure~\ref{fig:sim3} how $\hat p^*$ varies as $n$ increases. We consider vectors $kx$, for
$x=(120,35,22,55,30,18)$ and for $k=1,\dots,5$, with
$n=280k$.  
For each $k$, we let 
\[
p\in\left\{\tfrac{4\log n}{n},\,\tfrac{6\log n}{n},\,n^{-0.4}\right\}.
\]
For each $(k,p)$, we generate $N=10^4$ samples. 
We observe that $\hat p^*$ decreases for $p=4\log n/n$, stays around $0.78$ for $p=6\log n/n$, and increases rapidly toward $1$ for $p=n^{-0.4}$.

\input{fig7}

\section{Conclusions}

In this paper, we have shown that the problem of finding largest $2$-regular subgraphs in complete $S$-partite graphs can be solved efficiently, through the linear program~\eqref{eq:LP-y}. The extremal solutions of~\eqref{eq:LP-y} are integer valued, and an integer solution can be obtained in time $O(q^3)$, with $q=|V(S)|$. Moreover, if $x_i\ge 3$ for all $i$, then any integer solution of~\eqref{eq:LP-y} yields a solution to the original problem~\eqref{eq:mainproblem}. 
Furthermore, in Section~\ref{sec:ns}, we have demonstrated that with high probability, a random $S$-partite graph $G$ has a largest $2$-regular subgraph of the same order as its complete version does. 

\printbibliography

\end{document}

%% file: fig5.tex

\definecolor{myblue}{HTML}{1F77B4}
\definecolor{myorange}{HTML}{FF7F0E}
\definecolor{mygreen}{HTML}{2CA02C}
\definecolor{myred}{HTML}{D62728}
\definecolor{mybrown}{HTML}{8C564B}
\definecolor{mypurple}{HTML}{9467BD}

\newcommand{\pmfdata}{fig5.csv}

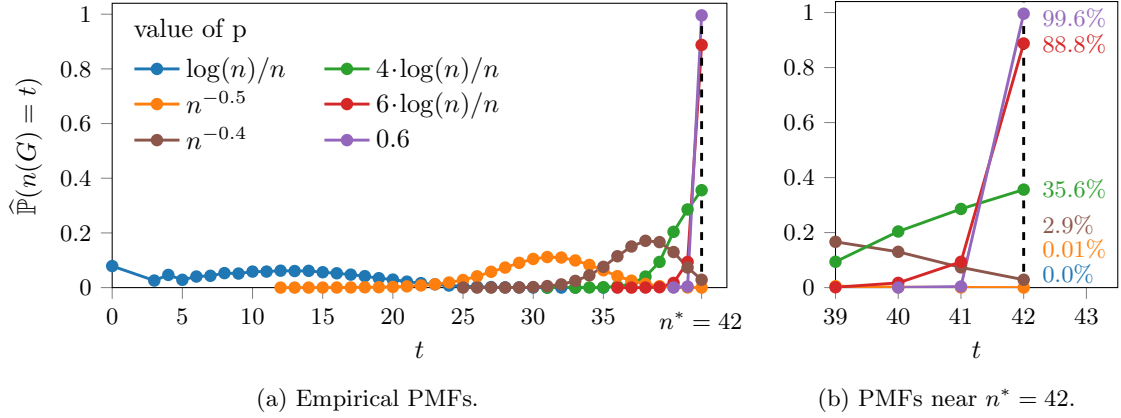
\begin{figure}[!t]
  \centering

  \begin{minipage}[t]{0.64\linewidth}
    \centering
    \begin{tikzpicture}
      \begin{axis}[
        width=\linewidth,
        height=0.55\linewidth,
        xmin=0, xmax=43.8,
        ymin=0, ymax=1.04,
        xlabel={$t$},
        ylabel={$\widehat{\mathbb{P}}(n(G)=t)$},
        xtick={0,5,10,15,20,25,30,35,42},
        xticklabels={0,5,10,15,20,25,30,35,},
        ytick={0,0.2,0.4,0.6,0.8,1.0},
        tick align=outside,
        tick pos=left,
        scaled ticks=false,
        axis line style={black},
        legend columns=2,
        legend style={
          draw=none,
          fill=none,
          at={(0.02,0.85)},
          anchor=north west,
          cells={anchor=west},
          /tikz/every even column/.append style={column sep=1.2em},
          font=\normalsize,
        },
        clip=false,
        tick label style={font=\small},
        label style={font=\normalsize},
      ]

        \node[anchor=north west, font=\normalsize] at (rel axis cs:0.02,0.98) {value of p};

        \addplot[color=myblue, mark=*, solid, line width=1.1pt, mark size=1.8pt]
          table[x=x_log, y=y_log, col sep=comma] {\pmfdata};
        \addlegendentry{$\log(n)/n$}

        \addplot[color=mygreen, mark=*, solid, line width=1.1pt, mark size=1.8pt]
          table[x=x_4log, y=y_4log, col sep=comma] {\pmfdata};
        \addlegendentry{$4\!\cdot\!\log(n)/n$}

        \addplot[color=myorange, mark=*, solid, line width=1.1pt, mark size=1.8pt]
          table[x=x_nhalf, y=y_nhalf, col sep=comma] {\pmfdata};
        \addlegendentry{$n^{-0.5}$}

        \addplot[color=myred, mark=*, solid, line width=1.1pt, mark size=1.8pt]
          table[x=x_6log, y=y_6log, col sep=comma] {\pmfdata};
        \addlegendentry{$6\!\cdot\!\log(n)/n$}

        \addplot[color=mybrown, mark=*, solid, line width=1.1pt, mark size=1.8pt]
          table[x=x_nfour, y=y_nfour, col sep=comma] {\pmfdata};
        \addlegendentry{$n^{-0.4}$}

        \addplot[color=mypurple, mark=*, solid, line width=1.1pt, mark size=1.8pt]
          table[x=x_const, y=y_const, col sep=comma] {\pmfdata};
        \addlegendentry{$0.6$}

        \addplot[black, dashed, line width=1.1pt, forget plot]
          coordinates {(42,0) (42,1.04)};

        \node[anchor=north, inner sep=1pt, font=\small] at (axis cs:41.95,0) [yshift=-8pt] {$n^*=42$};

      \end{axis}
    \end{tikzpicture}

    \vspace{1mm}
    {\small (a) Empirical PMFs.}
  \end{minipage}
  \hfill
    \begin{minipage}[t]{0.35\linewidth}
      \centering
      \begin{tikzpicture}
        \begin{axis}[
          width=\linewidth,
          height=1.01\linewidth,
          xmin=39, xmax=43.50,
          ymin=0, ymax=1.04,
          xlabel={$t$},
          ylabel={},
          xtick={39,40,41,42,43},
          xticklabels={39,40,41,42,43},
          ytick={0,0.2,0.4,0.6,0.8,1.0},
          tick align=outside,
          tick pos=left,
          scaled ticks=false,
          axis line style={black},
          tick label style={font=\small},
          label style={font=\normalsize},
        ]
    
          \addplot[color=myblue, mark=*, solid, line width=1.1pt, mark size=1.8pt,
            restrict x to domain=39:43]
            table[x=x_log, y=y_log, col sep=comma] {\pmfdata};
    
          \addplot[color=myorange, mark=*, solid, line width=1.1pt, mark size=1.8pt,
            restrict x to domain=39:43]
            table[x=x_nhalf, y=y_nhalf, col sep=comma] {\pmfdata};
    
          \addplot[color=mybrown, mark=*, solid, line width=1.1pt, mark size=1.8pt,
            restrict x to domain=39:43]
            table[x=x_nfour, y=y_nfour, col sep=comma] {\pmfdata};
    
          \addplot[color=mygreen, mark=*, solid, line width=1.1pt, mark size=1.8pt,
            restrict x to domain=39:43]
            table[x=x_4log, y=y_4log, col sep=comma] {\pmfdata};
    
          \addplot[color=myred, mark=*, solid, line width=1.1pt, mark size=1.8pt,
            restrict x to domain=39:43]
            table[x=x_6log, y=y_6log, col sep=comma] {\pmfdata};
    
          \addplot[color=mypurple, mark=*, solid, line width=1.1pt, mark size=1.8pt,
            restrict x to domain=39:43]
            table[x=x_const, y=y_const, col sep=comma] {\pmfdata};
    
          \addplot[black, dashed, line width=1.1pt, forget plot]
            coordinates {(42,0) (42,1.04)};
    
          \node[anchor=west, text=mypurple, font=\small] at (axis cs:42.15,0.9800) {99.6\%};
          \node[anchor=west, text=myred,    font=\small] at (axis cs:42.15,0.88800) {88.8\%};
          \node[anchor=west, text=mygreen,  font=\small] at (axis cs:42.15,0.36000) {35.6\%};
          \node[anchor=west, text=mybrown,  font=\small] at (axis cs:42.15,0.23000) {2.9\%};
          \node[anchor=west, text=myorange, font=\small] at (axis cs:42.15,0.14000) {0.01\%};
          \node[anchor=west, text=myblue,   font=\small] at (axis cs:42.15,0.05000) {0.0\%};
    
        \end{axis}
      \end{tikzpicture}
    
      \vspace{1mm}
      {\small (b) PMFs near $n^*=42$.}
    \end{minipage}

  \caption{Empirical PMFs $\widehat{\mathbb{P}}(n(G)=t)$ for the setting $x := (24,7,4,11,6,4)$ with $n = 56$. The dashed, vertical line is at $n^* = 42$. Values of $\hat p^*$ for different $p$ are reported.}
  \label{fig:sim1}
\end{figure}

%% file: fig6.tex

\definecolor{mygreen}{HTML}{2CA02C}
\definecolor{myred}{HTML}{D62728}
\definecolor{mybrown}{HTML}{8C564B}
\definecolor{mypurple}{HTML}{9467BD}

\newcommand{\pmfdatatwo}{fig6.csv}

\begin{figure}[!t]
  \centering

  \begin{minipage}[t]{0.64\linewidth}
    \centering
    \begin{tikzpicture}
      \begin{axis}[
        width=\linewidth,
        height=0.55\linewidth,
        xmin=196.8, xmax=210.85,
        ymin=0, ymax=1.04,
        xlabel={$t$},
        ylabel={$\widehat{\mathbb{P}}(n(G)=t)$},
        xtick={198,200,202,204,206,208,210},
        xticklabels={198,200,202,204,206,208,},
        ytick={0,0.2,0.4,0.6,0.8,1.0},
        tick align=outside,
        tick pos=left,
        scaled ticks=false,
        axis line style={black},
        legend columns=2,
        legend style={
          draw=none,
          fill=none,
          at={(0.03,0.87)},
          anchor=north west,
          cells={anchor=west},
          /tikz/every even column/.append style={column sep=1.25em},
          font=\normalsize,
        },
        clip=false,
        tick label style={font=\small},
        label style={font=\normalsize},
      ]

        \node[anchor=north west, font=\normalsize] at (rel axis cs:0.03,0.96) {value of p};

        \addplot[color=mygreen, mark=*, solid, line width=1.1pt, mark size=1.8pt]
          table[x=t, y=y_green, col sep=comma] {\pmfdatatwo};
        \addlegendentry{$4\!\cdot\!\log(n)/n$}

        \addplot[color=mybrown, mark=*, solid, line width=1.1pt, mark size=1.8pt,
        restrict x to domain=202:210]
          table[x=t, y=y_brown, col sep=comma] {\pmfdatatwo};
        \addlegendentry{$n^{-0.4}$}

        \addplot[color=myred, mark=*, solid, line width=1.1pt, mark size=1.8pt,
        restrict x to domain=205:210]
          table[x=t, y=y_red, col sep=comma] {\pmfdatatwo};
        \addlegendentry{$6\!\cdot\!\log(n)/n$}

        \addplot[color=mypurple, mark=*, only marks, mark size=2.3pt]
          coordinates {(210,1.0)};
        \addlegendentry{$0.6$}

        \addplot[black, dashed, line width=1.1pt, forget plot]
          coordinates {(210,0) (210,1.04)};

        \node[anchor=north, inner sep=1pt, font=\small]
          at (axis cs:210,0) [yshift=-8pt] {$n^*=210$};

      \end{axis}
    \end{tikzpicture}

    \vspace{1mm}
    {\small (a) Empirical PMFs.}
  \end{minipage}
  \hfill
  \begin{minipage}[t]{0.35\linewidth}
    \centering
    \begin{tikzpicture}
      \begin{axis}[
        width=\linewidth,
        height=1.01\linewidth,
        xmin=207, xmax=211.7,
        ymin=0, ymax=1.04,
        xlabel={$t$},
        ylabel={},
        xtick={207,208,209,210,211},
        xticklabels={207,208,209,210,211},
        ytick={0,0.2,0.4,0.6,0.8,1.0},
        tick align=outside,
        tick pos=left,
        scaled ticks=false,
        axis line style={black},
        clip=false,
        tick label style={font=\small},
        label style={font=\normalsize},
      ]

        \addplot[color=mygreen, mark=*, solid, line width=1.1pt, mark size=1.8pt,
          restrict x to domain=207:210]
          table[x=t, y=y_green, col sep=comma] {\pmfdatatwo};

        \addplot[color=mybrown, mark=*, solid, line width=1.1pt, mark size=1.8pt,
          restrict x to domain=207:210]
          table[x=t, y=y_brown, col sep=comma] {\pmfdatatwo};

        \addplot[color=myred, mark=*, solid, line width=1.1pt, mark size=1.8pt,
          restrict x to domain=207:210]
          table[x=t, y=y_red, col sep=comma] {\pmfdatatwo};

        \addplot[color=mypurple, mark=*, only marks, mark size=2.3pt]
          coordinates {(210,1.0)};

        \addplot[black, dashed, line width=1.1pt, forget plot]
          coordinates {(210,0) (210,1.04)};

        \node[anchor=west, text=mypurple, font=\small] at (axis cs:210.12,0.9900) {100.0\%};
        \node[anchor=west, text=myred,    font=\small] at (axis cs:210.12,0.8028) {80.3\%};
        \node[anchor=west, text=mybrown,  font=\small] at (axis cs:210.12,0.5947) {59.5\%};
        \node[anchor=west, text=mygreen,  font=\small] at (axis cs:210.12,0.1486) {14.9\%};

      \end{axis}
    \end{tikzpicture}

    \vspace{1mm}
    {\small (b) PMFs near $n^*=210$.}
  \end{minipage}

  \caption{Empirical PMFs $\widehat{\mathbb{P}}(n(G)=t)$ for the case $x=(120,35,22,55,30,18)$ with $n=280$. The dashed, vertical line marks $n^*=210$, and the labels report $\hat p^*$ for different $p$.}
  \label{fig:sim2}
\end{figure}
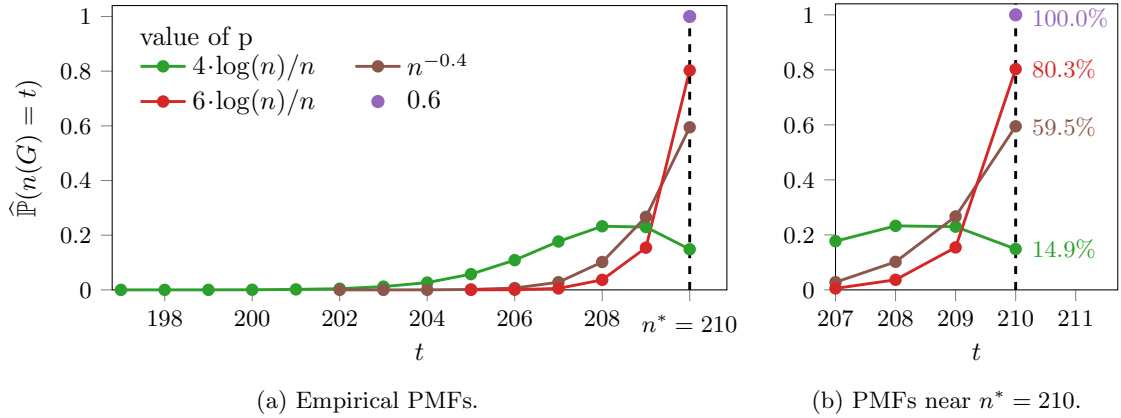

%% file: fig7.tex

\definecolor{mygreen}{HTML}{2CA02C}
\definecolor{myred}{HTML}{D62728}
\definecolor{mybrown}{HTML}{8C564B}

\newcommand{\pmfdatathree}{fig7.csv}

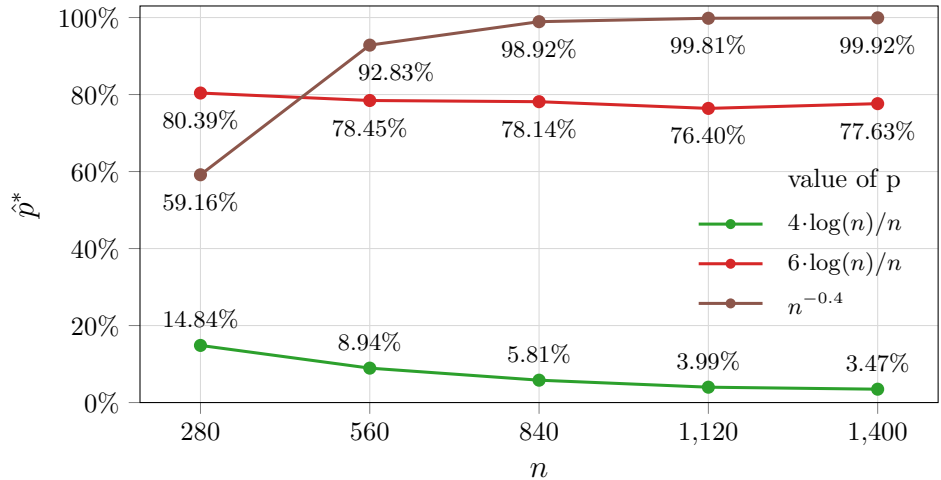
\begin{figure}[!t]
  \centering
  \begin{tikzpicture}
    \begin{axis}[
      width=0.8\linewidth,
      height=0.45\linewidth,
      xmin=180, xmax=1500,
      ymin=0, ymax=1.03,
      xlabel={$n$},
      ylabel={$\hat p^*$},
      label style={font=\large},
      tick label style={font=\normalsize},
      xtick={280,560,840,1120,1400},
      xticklabels={280,560,840,{1,120},{1,400}},
      ytick={0,0.2,0.4,0.6,0.8,1.0},
      yticklabels={0\%,20\%,40\%,60\%,80\%,100\%},
      tick align=outside,
      tick pos=left,
      scaled ticks=false,
      axis line style={black},
      grid=major,
      grid style={draw=gray!30},
      clip=false,
    ]

      \addplot[color=mygreen, mark=*, solid, line width=1.15pt, mark size=1.9pt]
        table[x=n, y=y_green, col sep=comma] {\pmfdatathree};

      \addplot[color=myred, mark=*, solid, line width=1.15pt, mark size=1.9pt]
        table[x=n, y=y_red, col sep=comma] {\pmfdatathree};

      \addplot[color=mybrown, mark=*, solid, line width=1.15pt, mark size=1.9pt]
        table[x=n, y=y_brown, col sep=comma] {\pmfdatathree};

      \node[font=\small, anchor=south] at (axis cs:280,0.1484)  [yshift=3pt] {14.84\%};
      \node[font=\small, anchor=south] at (axis cs:560,0.0894)  [yshift=3pt] {8.94\%};
      \node[font=\small, anchor=south] at (axis cs:840,0.0581)  [yshift=3pt] {5.81\%};
      \node[font=\small, anchor=south] at (axis cs:1120,0.0399) [yshift=3pt] {3.99\%};
      \node[font=\small, anchor=south] at (axis cs:1400,0.0347) [yshift=3pt] {3.47\%};

      \node[font=\small, anchor=north] at (axis cs:280,0.8039)  [yshift=-3pt] {80.39\%};
      \node[font=\small, anchor=north] at (axis cs:560,0.7845)  [yshift=-3pt] {78.45\%};
      \node[font=\small, anchor=north] at (axis cs:840,0.7814)  [yshift=-3pt] {78.14\%};
      \node[font=\small, anchor=north] at (axis cs:1120,0.7640) [yshift=-3pt] {76.40\%};
      \node[font=\small, anchor=north] at (axis cs:1400,0.7763) [yshift=-3pt] {77.63\%};

      \node[font=\small, anchor=north] at (axis cs:280,0.5916)  [yshift=-3pt] {59.16\%};
      \node[font=\small, anchor=north] at (axis cs:560,0.9283)  [xshift=10pt,yshift=-3pt] {92.83\%};
      \node[font=\small, anchor=north] at (axis cs:840,0.9892)  [yshift=-3pt] {98.92\%};
      \node[font=\small, anchor=north] at (axis cs:1120,0.9981) [yshift=-3pt] {99.81\%};
      \node[font=\small, anchor=north] at (axis cs:1400,0.9992) [yshift=-3pt] {99.92\%};

      \node[anchor=west, font=\normalsize] at (axis description cs:0.80,0.56) {value of p};

      \draw[mygreen, line width=1.15pt] (axis description cs:0.70,0.45) -- (axis description cs:0.78,0.45);
      \fill[mygreen] (axis description cs:0.74,0.45) circle (1.9pt);
      \node[anchor=west, font=\small] at (axis description cs:0.80,0.45) {$4\!\cdot\!\log(n)/n$};

      \draw[myred, line width=1.15pt] (axis description cs:0.70,0.35) -- (axis description cs:0.78,0.35);
      \fill[myred] (axis description cs:0.74,0.35) circle (1.9pt);
      \node[anchor=west, font=\small] at (axis description cs:0.80,0.35) {$6\!\cdot\!\log(n)/n$};

      \draw[mybrown, line width=1.15pt] (axis description cs:0.70,0.25) -- (axis description cs:0.78,0.25);
      \fill[mybrown] (axis description cs:0.74,0.25) circle (1.9pt);
      \node[anchor=west, font=\small] at (axis description cs:0.80,0.25) {$n^{-0.4}$};

    \end{axis}
  \end{tikzpicture}
  \caption{Plots of $\hat p^*$ for different $(k,p)$, with $x = k(120,35,22,55,30,18)$ and $n = 280k$, for $k=1,\ldots,5$.}
  \label{fig:sim3}
\end{figure}